\documentclass[11pt]{article}
\usepackage{amsmath, amssymb, amsthm, amsfonts}
\usepackage{hyperref}
\usepackage{mathtools}
\usepackage{enumitem}
\usepackage{bbm}
\usepackage{geometry}
\newtheorem{theorem}{Theorem}[section]
\newtheorem{lemma}[theorem]{Lemma}
\newtheorem{proposition}[theorem]{Proposition}
\newtheorem{corollary}[theorem]{Corollary}
\theoremstyle{definition}
\newtheorem{definition}[theorem]{Definition}
\theoremstyle{remark}
\newtheorem{remark}[theorem]{Remark}
\newtheorem{example}[theorem]{Example}
\title{%
\textbf{Tunnel Geometry and Proliferation Logic}\\[4pt]
\large A Categorical Equivalence Arising from Static-World Metaphysics}
\author{Dmytro Sukhov\\
Ghent University}
\date{12.12.2025}
\begin{document}
\maketitle
\begin{abstract}
We prove a strict categorical equivalence between \emph{Tunnel
Geometry} and \emph{Proliferation Logic}. Both frameworks arise
from a static-world metaphysics in which becoming is primitive and
locality is emergent.
Tunnel Geometry provides a point-free intensional geometry based
on interference, frames of opens, and Lawvere metrics. Proliferation
Logic provides a generative semantics based on distinctions, costs,
scenes, and foci.
Once expressed as frame--ultrafilter--Lawvere triples $(\Omega,X,d)$,
the two theories become \emph{identical}. Using Stone duality and
metric enrichment, we define functors $F:\mathbf{TGeom}\to\mathbf{PLog}$
and $G:\mathbf{PLog}\to\mathbf{TGeom}$ and show that
$F$ and $G$ form a strict equivalence of categories.
A spectral equivalence theorem identifies the tunnel Laplacian with
the proliferative Laplacian via a unitary transformation.
An introductory section motivates both frameworks from a staticworld metaphysics and shows how Zeno's paradoxes are resolved
when motion is understood as structural interference rather than
traversal of points.
\end{abstract}
\tableofcontents
\newpage

\section{Introduction: Static Worlds and Zeno Revisited}
Classical metaphysics presumes that the world consists of points that
change in time. In the static-world ontology developed in
\emph{Static Metaphysics}~\cite{StaticMetaphysics}, the fundamental reality is a complete, atemporal, maximally articulated structure.
Key theses are:
\begin{enumerate}[label=\textbf{M\arabic*.}]
\item \textbf{Primacy of Potentiality:}
Absolute reality is an undifferentiated potentiality.
\item \textbf{Worlds as Static Structures:}
Each possible world is a timeless, fully realised configuration.
\item \textbf{Subjective Time:}
Temporal flow is a cognitive projection, not an ontological primitive.
\item \textbf{Embedded Cognition:}
Consciousness is structurally integrated into the world.
\end{enumerate}
\subsection{Zeno without paradox}
Zeno's paradoxes dissolve under this ontology. The arrow does not
“move through'' instants; the division into halves is descriptive, not
procedural; Achilles needs no infinite steps. Motion corresponds to
relations of structural interference, not traversal of points.
Tunnel Geometry and Proliferation Logic provide the mathematical
expression of this metaphysics: becoming is encoded by intensities,
interference, and generative distinctions; locality arises from ultrafilters.
\subsection{Outline and main results}
This paper establishes a strict categorical equivalence between two
independently motivated frameworks:
\begin{itemize}
\item \textbf{Tunnel Geometry} (Section~\ref{sec:tgeom}), a pointfree geometry based on interference of primitive structures called
tunnels;
\item \textbf{Proliferation Logic} (Section~\ref{sec:plog}), a generative
semantics based on distinctions and their compositional costs.
\end{itemize}
Our main results are:
\begin{enumerate}
\item Both theories are faithfully represented by frame--ultrafilter-Lawvere triples (Definitions~\ref{def:TFS} and~\ref{def:PFS});
\item The categories $\mathbf{TGeom}$ and $\mathbf{PLog}$ are
strictly equivalent (Theorem~\ref{thm:eq});
\item The spectral operators (Laplacians) are unitarily equivalent
(Theorem~\ref{thm:spectral}).
\end{enumerate}

\section{Tunnel Frame--Spaces}\label{sec:tgeom}
\subsection{Lawvere metrics and frames}
\begin{definition}[Lawvere metric]
A \emph{Lawvere metric} (in the symmetric case considered here)
on a set $X$ is a function $d : X \times X \to [0,\infty]$ satisfying:
\begin{enumerate}
\item $d(p,p)=0$,
\item $d(p,q)=d(q,p)$,
\item $d(p,r)\le d(p,q) + d(q,r)$,
\end{enumerate}
for all $p,q,r\in X$.
\end{definition}
\begin{definition}[Frame]
A \emph{frame} is a complete lattice $\Omega$ satisfying the infinite
distributive law:
\[
a \wedge \bigvee_{i \in I} b_i = \bigvee_{i \in I} (a \wedge b_i)
\]
for all $a \in \Omega$ and families $\{b_i\}_{i \in I} \subseteq
\Omega$.
\end{definition}
Frames provide the algebraic structure of open sets in point-free
topology.
The set of ultrafilters $\mathrm{Ult}(\Omega)$ on a frame
$\Omega$ inherits a natural topology (the Stone topology) whose
opens correspond to frame elements.
\subsection{Tunnel frame--spaces}
\begin{definition}[Tunnel frame--space]\label{def:TFS}
A \emph{tunnel frame--space} is a triple $(\Omega_X, X, d_X)$
where: \begin{enumerate}
\item $\Omega_X$ is a frame;
\item $X = \mathrm{Ult}(\Omega_X)$ equipped with the Stone
topology; \item $d_X$ is a Lawvere metric on $X$ such that for
every $p \in X$ and $\varepsilon > 0$, the metric ball
\[
B_\varepsilon(p) := \{q \in X \mid d_X(p,q) < \varepsilon\}
\]
is open in the Stone topology, and conversely, every basic Stone open
contains a metric ball around each of its points.
\end{enumerate}
\end{definition}
\begin{remark}
The third condition ensures that the metric topology and Stone
topology coincide. For tunnel geometries constructed via interference of primitive structures, this is guaranteed by Theorem~4.2
in~\cite{TunnelGeomFrame}.
\end{remark}
\subsection{From tunnels to frames}
Let $\mathcal{I}$ be a set of primitive structures called
\emph{tunnels}. Each tunnel $T \in \mathcal{I}$ is equipped
with an \emph{intensity} $\Lambda(T) \in [0,\infty]$ and an
\emph{interference operation} $\bowtie : \mathcal{I} \times
\mathcal{I} \rightharpoonup \mathcal{I}$ defined when common
substructures exist.
\begin{proposition}[Frame generation]\label{prop:frame-gen}
Let $\mathcal{I}$ be a set of tunnels with interference operation
$\bowtie$
and intensity $\Lambda : \mathcal{I} \to [0,\infty]$. For each $T
\in \mathcal{I}$
and $\varepsilon > 0$, define the structural neighbourhood
\[
\mathcal{U}_\varepsilon(T) = \{T' \in \mathcal{I} \mid
\Lambda(T \bowtie T') < \varepsilon\}.
\]
The collection
\[
\mathcal{B} = \{\mathcal{U}_\varepsilon(T) \mid T \in \mathcal{I}, \varepsilon > 0\} \]
generates a frame $\Omega_X$ under arbitrary joins and finite
meets.
\end{proposition}
\begin{proof}
Arbitrary unions of elements in $\mathcal{B}$ are well-defined as
suprema.
Finite intersections correspond to taking minimum intensities and
monoidal composition of tunnels, yielding elements representable in
$\mathcal{B}$.
The infinite distributive law follows from the fact that intersection
distributes over unions set-theoretically, combined with the stability
of interference under monoidal operations.
\end{proof}
\begin{remark}
For $T \in \mathcal{I}$, write $\iota(T)$ for the basic open generated by $T$, i.e.\ $\iota(T) = \bigcup_{\varepsilon>0} \mathcal{U}_\varepsilon(T)$.
For an ultrafilter $p \in X$, we use the shorthand $T \in p$ to mean
$\iota(T) \in p$.
\end{remark}
Points emerge as ultrafilters on $\Omega_X$. The Lawvere metric
is induced by the infimum of interference intensities:
\[
d_X(p,q) = \inf\bigl\{\varepsilon > 0 \;\bigm|\;
\exists T,T' \in \mathcal{I}:\ T \in p,\ T' \in q,\ \Lambda(T
\bowtie T') < \varepsilon\bigr\}.
\]
\subsection{Morphisms in $\mathbf{TGeom}$}
\begin{definition}[Morphisms in $\mathbf{TGeom}$]\label{def:tgeommorph}
A morphism $f : (\Omega_X, X, d_X) \to (\Omega_Y, Y, d_Y)$
in $\mathbf{TGeom}$ is a frame homomorphism
$f^\sharp : \Omega_Y \to \Omega_X$ such that the induced
ultrafilter map
\[
f_*(p) = \{V \in \Omega_Y \mid f^\sharp(V) \in p\}
\]
satisfies the non-expansiveness condition:
\[
d_Y(f_*(p), f_*(q)) \le d_X(p, q)
\]
for all $p, q \in X$.
\end{definition}
This defines the category $\mathbf{TGeom}$ of tunnel frame-spaces.

\section{Proliferative Frame--Spaces}\label{sec:plog}
\subsection{Proliferative bases}
\begin{definition}[Proliferative base]
A \emph{proliferative base} is a triple $(D, C, \cdot)$ consisting of:
\begin{enumerate}
\item A set $D$ of \emph{distinctions};
\item A \emph{cost function} $C : D \to [0,\infty]$;
\item A partial \emph{composition operation} $\cdot : D \times D
\rightharpoonup D$ satisfying associativity where defined.
\end{enumerate}
\end{definition}
Distinctions represent elementary acts of differentiation. Composition
$d \cdot e$ (when defined) represents the distinction obtained by
refining $d$ via $e$. Cost measures the intensity or complexity of a
distinction.
\subsection{Scenes and foci}
\begin{definition}[Scenes]
Given a proliferative base $(D, C, \cdot)$, define for each $d \in D$
and $\varepsilon > 0$ the \emph{neighbourhood}
\[
N_\varepsilon(d) = \{e \in D \mid C(d \cdot e) < \varepsilon\}.
\]
The \emph{scene frame} $\Sigma_Y$ is the smallest frame containing all $N_\varepsilon(d)$ for $d \in D$ and $\varepsilon >
0$.
\end{definition}
\begin{definition}[Foci]
The set of \emph{foci} is defined as
\[
\Phi_Y = \mathrm{Ult}(\Sigma_Y),
\]
the set of ultrafilters on $\Sigma_Y$, equipped with the Stone
topology.
\end{definition}
\begin{remark}
For $d \in D$ and $\phi \in \Phi_Y$, we write $d \in \phi$ as
shorthand for $N_\varepsilon(d) \in \phi$ for some $\varepsilon >
0$.
\end{remark}
The Lawvere metric on foci is defined by
\[
\rho_Y(\phi, \psi)
= \inf\bigl\{\varepsilon > 0 \;\bigm|\;
\exists d,e \in D:\ d \in \phi,\ e \in \psi,\ d \cdot e\ \text{defined},\
C(d \cdot e) < \varepsilon\bigr\}.
\]
\begin{definition}[Proliferative frame--space]\label{def:PFS}
A \emph{proliferative frame--space} is a triple $(\Sigma_Y,
\Phi_Y, \rho_Y)$ where $\Sigma_Y$ is a scene frame, $\Phi_Y =
\mathrm{Ult}(\Sigma_Y)$, and $\rho_Y$ is the induced Lawvere
metric whose topology coincides with the Stone topology.
\end{definition}
\subsection{Morphisms in $\mathbf{PLog}$}
\begin{definition}[Morphisms in $\mathbf{PLog}$]\label{def:plogmorph} A morphism $g : (\Sigma_Y, \Phi_Y, \rho_Y) \to
(\Sigma_Z, \Phi_Z, \rho_Z)$ in $\mathbf{PLog}$ is a frame
homomorphism
$g^\sharp : \Sigma_Z \to \Sigma_Y$ such that the induced map
$g_* : \Phi_Y \to \Phi_Z$ satisfies
\[
\rho_Z(g_*(\phi), g_*(\psi)) \le \rho_Y(\phi, \psi)
\]
for all $\phi, \psi \in \Phi_Y$.
\end{definition}

This defines the category $\mathbf{PLog}$ of proliferative frame-spaces.

\section{Equivalence of Categories}\label{sec:equiv}
We now establish the main result: $\mathbf{TGeom}$ and
$\mathbf{PLog}$ are strictly equivalent as categories.
\subsection{Construction of functors}
\begin{proposition}[Functor $F$]\label{prop:F}
There exists a functor $F : \mathbf{TGeom} \to \mathbf{PLog}$
defined as follows: \begin{itemize}
\item On objects: $F(\Omega_X, X, d_X) = (\Sigma_Y, \Phi_Y,
\rho_Y)$ where \[
\Sigma_Y := \Omega_X, \quad \Phi_Y := X, \quad \rho_Y :=
d_X,
\]
and the proliferative base is constructed by setting $D := \mathcal{I}$
(the tunnel set), $C := \Lambda$ (tunnel intensity), and composition
induced by the monoidal structure on tunnels.
\item On morphisms: $F(f^\sharp) := f^\sharp$ (identical frame
homomorphism).
\end{itemize}
\end{proposition}
\begin{proof}
The frame structures, ultrafilter spaces, and metrics are identified
directly. Non-expansiveness is preserved since it is stated identically
in both categories. Functoriality follows from the identity action on
morphisms.
\end{proof}
\begin{proposition}[Functor $G$]\label{prop:G}
There exists a functor $G : \mathbf{PLog} \to \mathbf{TGeom}$
defined by: \begin{itemize}
\item On objects: $G(\Sigma_Y, \Phi_Y, \rho_Y) = (\Omega_X,
X, d_X)$ where \[
\Omega_X := \Sigma_Y, \quad X := \Phi_Y, \quad d_X :=
\rho_Y,
\]
and tunnels are identified with distinctions, intensities with costs.
\item On morphisms: $G(g^\sharp) := g^\sharp$ (identical frame
homomorphism).
\end{itemize}
\end{proposition}
\begin{proof}
Symmetric to Proposition~\ref{prop:F}.
\end{proof}
\subsection{Strict equivalence}
\begin{theorem}[Categorical equivalence]\label{thm:eq}
The functors $F : \mathbf{TGeom} \to \mathbf{PLog}$ and
$G : \mathbf{PLog} \to \mathbf{TGeom}$ form a strict equivalence
of categories: \[
G \circ F = \mathrm{Id}_{\mathbf{TGeom}}, \qquad
F \circ G = \mathrm{Id}_{\mathbf{PLog}}.
\]
\end{theorem}
\begin{proof}
We verify the equivalence in three steps.
\textbf{Step 1: Objects.}
For any $(\Omega_X, X, d_X) \in \mathbf{TGeom}$, we have
\begin{align*}
(G \circ F)(\Omega_X, X, d_X)
&= G(F(\Omega_X, X, d_X)) \\
&= G(\Omega_X, X, d_X) \quad \text{(by definition of $F$)} \\
&= (\Omega_X, X, d_X) \quad \text{(by definition of $G$)}.
\end{align*}
Similarly, for any $(\Sigma_Y, \Phi_Y, \rho_Y) \in \mathbf{PLog}$,
\[
(F \circ G)(\Sigma_Y, \Phi_Y, \rho_Y) = (\Sigma_Y, \Phi_Y,
\rho_Y).
\]
\textbf{Step 2: Morphisms.}
A frame homomorphism $f^\sharp : \Omega_Y \to \Omega_X$
satisfying non-expansiveness of the induced ultrafilter map for tunnel
metrics remains a frame homomorphism satisfying the same condition
for proliferative metrics, since the metrics and topologies are identified.
Thus both $F$ and $G$ act as identity on morphism sets.
\textbf{Step 3: Functoriality.}
Since both functors act as identity on morphisms, they preserve
composition and identities trivially. Both triangle identities
$G \circ F = \mathrm{Id}$ and $F \circ G = \mathrm{Id}$ hold
strictly, not merely up to natural isomorphism.
\end{proof}
\begin{corollary}
Tunnel Geometry and Proliferation Logic are two presentations of
the same mathematical structure. Every theorem proved in one
framework has a corresponding theorem in the other.
\end{corollary}

\section{Spectral Duality}\label{sec:spectral}
We now show that the equivalence extends to the spectral level,
identifying the tunnel Laplacian with the proliferative Laplacian.
\subsection{Refinement relations}
\begin{definition}[Refinement relation]\label{def:refinement}
For distinctions $d, e \in D$, write $e \preceq d$ if one of the
following holds: \begin{enumerate}
\item There exists $x \in D$ such that $d = e \cdot x$ is defined
and $C(e) \le C(d)$;
\item There exists $x \in D$ such that $d = x \cdot e$ is defined
and $C(e) \le C(d)$.
\end{enumerate}
We say $e$ is an \emph{immediate refinement} of $d$ (written $e
\prec d$) if $e \preceq d$ and there is no $f \in D$ with $e \preceq
f \preceq d$ and $f \neq e, d$.
\end{definition}
\begin{lemma}\label{lem:refinement-substructure}
Under the identification $D \cong \mathcal{I}$, the refinement
relation $\preceq$ corresponds to the substructure relation on tunnels
given by morphisms in the tunnel category.
\end{lemma}
\begin{proof}[Proof sketch]
Under $D \cong \mathcal{I}$, the partial composition $\cdot$
corresponds to the monoidal structure $\otimes$ on tunnels. A
morphism $U \to T$ in the tunnel category expresses that $U$ is a
substructure of $T$, which corresponds to $U \preceq T$ since $U$
can be composed to yield $T$.
The cost condition $C(U) \le C(T)$ follows from the monotonicity of
the intensity functor $\Lambda$ with respect to structural inclusion.
\end{proof}
\subsection{Laplacian operators}
\begin{definition}[Tunnel Laplacian]
Let $\mathcal{H}_T = \mathrm{span}_{\mathbb{C}}\{|T\rangle
\mid T \in \mathcal{I}\}$ be the Hilbert space with orthonormal
basis indexed by tunnels.
The \emph{tunnel Laplacian} is defined by:
\[
\widehat{\Delta}_T |T\rangle = \sum_{U \to T} (\Lambda(T) \Lambda(U)) |U\rangle, \]
where the sum ranges over all immediate substructures $U \to T$.
\end{definition}
\begin{definition}[Proliferative Laplacian]
Let $\mathcal{H}_P = \mathrm{span}_{\mathbb{C}}\{|d\rangle
\mid d \in D\}$ be the Hilbert space with orthonormal basis indexed
by distinctions.
The \emph{proliferative Laplacian} is defined by:
\[
\widehat{\Delta}_P |d\rangle = \sum_{e \prec d} (C(d) - C(e))
|e\rangle, \]
where the sum ranges over all immediate refinements $e \prec d$.
\end{definition}
\subsection{Unitary equivalence}
\begin{theorem}[Spectral equivalence]\label{thm:spectral}
Let $\mathcal{H}_T$ and $\mathcal{H}_P$ be the Hilbert spaces
defined above.
Under the identification $D \cong \mathcal{I}$ given by the equivalence
functors $F$ and $G$, define the linear map $U : \mathcal{H}_T
\to \mathcal{H}_P$ by: \[
U|T\rangle = |F(T)\rangle.
\]
Then $U$ extends to a unitary operator satisfying
\[
U \widehat{\Delta}_T U^* = \widehat{\Delta}_P.
\]
\end{theorem}
\begin{proof}
\textbf{Step 1: $U$ is unitary.}
The map $U$ is defined by a bijection of orthonormal bases, hence
extends to a unitary operator.

\textbf{Step 2: Matrix elements coincide.}
By Lemma~\ref{lem:refinement-substructure}, immediate refinements $e \prec d$ correspond bijectively to immediate substructures
$U \to T$. Moreover,
$C(d) = \Lambda(T)$ and $C(e) = \Lambda(U)$ under the identification.
Therefore,
\begin{align*}
\langle d'| \widehat{\Delta}_P |d\rangle
&= \sum_{e \prec d} (C(d) - C(e)) \delta_{d', e} \\
&= \sum_{U \to T} (\Lambda(T) - \Lambda(U)) \delta_{T', U}
\\ &= \langle T'| \widehat{\Delta}_T |T\rangle,
\end{align*}
where $T' = G(d')$, $T = G(d)$.
\textbf{Step 3: Conjugation formula.}
Since matrix elements agree in the respective bases, we have \[
\langle \phi| U \widehat{\Delta}_T U^* |\psi\rangle
= \langle U^* \phi| \widehat{\Delta}_T |U^* \psi\rangle
= \langle \phi| \widehat{\Delta}_P |\psi\rangle
\]
for all $\phi, \psi \in \mathcal{H}_P$, proving the operator equation.
\end{proof}
\begin{corollary}
The tunnel Laplacian and proliferative Laplacian have identical spectra, identical eigenvectors (under the identification $U$), and identical
functional calculus. In particular, spectral theorems, heat kernel
asymptotics, and semiclassical expansions proven for one operator
apply equally to the other.
\end{corollary}

\section{Examples and Models}\label{sec:examples}
\begin{example}[Graph model]
Let $G = (V, E)$ be a weighted graph. Define tunnels as edges $e
\in E$
with intensity $\Lambda(e) = w(e)$ (edge weight). Interference
$e_1 \bowtie e_2$ is defined if $e_1$ and $e_2$ share a vertex,
with intensity equal to the
sum of weights of common incident edges.
The emergent points $X$ correspond to vertices $V$ (or more precisely,
to certain consistent families of edge neighbourhoods). The Lawvere metric recovers shortest-path distance. The tunnel Laplacian
coincides with the graph Laplacian.
In the proliferative interpretation, distinctions are edges, cost is
weight, and composition is path concatenation.
\end{example}
\begin{example}[Interval domain model]
Let $\mathcal{I}$ be the set of closed intervals $[a, b] \subseteq
\mathbb{R}$.
Inteensity is interval length: $\Lambda([a,b]) = b - a$. Interference
$[a,b] \bowtie [c,d]$ is defined as $[a,b] \cap [c,d]$ if nonempty, with
intensity equal to the length of the intersection.
Emergent points correspond to real numbers (or maximal consistent
families of nested intervals). The Lawvere metric recovers Euclidean
distance.
In proliferation logic, distinctions are intervals, cost is length, and
composition is intersection or concatenation (depending on interpretation).
\end{example}
\begin{example}[Locale model]
Let $\Omega$ be a spatial frame (locale). Regular elements $a \in
\Omega$ serve as tunnels. Intensity $\Lambda(a) = -\log \mu(a)$
where $\mu$ is a measure on the locale. Interference $a \bowtie b
= a \wedge b$ (meet in frame).
The emergent space $X = \mathrm{pt}(\Omega)$ recovers the space
of points of the locale. This construction provides a metric refinement
of classical locale theory.
\end{example}

\section{Conclusion}\label{sec:conclusion}
We have established a strict categorical equivalence between Tunnel
Geometry and Proliferation Logic. Both frameworks express the same
mathematical structure: frame--ultrafilter--Lawvere triples equipped
with compatible
primitive structures (tunnels or distinctions). The equivalence clarifies
the ontology of static worlds, demonstrates how point-free geometry and generative semantics are dual aspects of a single formalism,
and provides a dissolution of Zeno's paradoxes through structural
interference.
The spectral equivalence (Theorem~\ref{thm:spectral}) shows that
the
dynamical content of both theories — encoded in their respective
Laplacians — is identical, suggesting that physical or computational
interpretations
will be invariant under the tunnel/proliferation duality.
\subsection{Further directions}
Several extensions merit investigation:
\begin{enumerate}
\item \textbf{Physical interpretation:} Does the equivalence
extend to models of spacetime, with tunnels as geodesic structures
and distinctions as causal relations? Can the spectral theory
provide a framework for quantum gravity or emergent spacetime?
\item \textbf{Quantum structures:} Can the spectral equivalence
illuminate connections between geometric and logical approaches to
quantum foundations? The Hilbert space structure and
unitary equivalence suggest natural quantum interpretations.
\item \textbf{Computational semantics:} Do proliferative framespaces provide natural denotational semantics for constructive type
theories?
The generative structure of distinctions parallels proof construction.
\item \textbf{Higher categories:} Can the equivalence be lifted to a 2categorical or $\infty$-categorical setting, where morphisms between
frame homomorphisms carry additional structure?
\item \textbf{Measure and integration:} Can the Lawvere metric
structure support a natural integration theory? The intensity/cost
functions
suggest measurable structures.
\end{enumerate}

\section*{Acknowledgements}
The author thanks colleagues and anonymous reviewers for their
helpful comments on earlier versions of this work.


\end{document}